\documentclass[preprint,12pt]{elsarticle}

\usepackage{amssymb}
\usepackage{amsmath}
\usepackage{mathtools}
\usepackage{amsfonts}
\usepackage{amsthm}
\usepackage{float}
\usepackage{times}
\usepackage{tikz-cd}
\usetikzlibrary{arrows}

\floatstyle{boxed}
\newtheorem{theorem}{Theorem}[section]
\newtheorem{proposition}[theorem]{Proposition}
\newtheorem{lemma}[theorem]{Lemma}
\newtheorem{corollary}[theorem]{Corollary}
\theoremstyle{definition}
\newtheorem{example}[theorem]{Example}

\newcommand{\defiff}{\: :\Longleftrightarrow\:}
\newcommand{\f}{\varphi}
\newcommand{\p}{\psi}

\newcommand{\eps}{\varepsilon}
\newcommand{\e}{{\rm e}}
\newcommand{\jn}{\vee}
\newcommand{\mt}{\wedge}
\newcommand{\m}{\mathbf} 
\newcommand{\N}{{\mathbb N}}
\newcommand{\Z}{{\mathbb Z}}
\newcommand{\Q}{{\mathbb Q}}

\newcommand{\iv}[1]{{#1}^{-1}}

\newcommand{\Si}{\mathrm{\Sigma}}
\newcommand{\Om}{\mathrm{\Omega}}
\newcommand{\vty}[1]{\mathsf{#1}}
\newcommand{\eq}{\approx}
\newcommand{\idq}[2]{\tfrac{#1}{#2}}
\DeclareMathOperator{\Sub}{is}
\newcommand{\set}[1]{\{ #1 \}}

\newcommand\nbd[1]{\protect\nobreakdash#1\hspace{0pt}}

\begin{document}

\begin{frontmatter}

\title{From distributive $\ell$\nbd{-}monoids to $\ell$\nbd{-}groups, and back again}

\author{Almudena Colacito\fnref{thanks1}}
\address{Laboratoire J. A. Dieudonn{\'e}, Universit{\'e} C{\^o}te d'Azur, France}
\ead{almudena.colacito@unice.fr}
\fntext[thanks1]{Supported by Swiss National Science Foundation (SNF) grant  No. P2BEP2\textunderscore195473.}	

\author{Nikolaos Galatos}
\address{Department of Mathematics, University of Denver, 2390 S. York St. Denver, CO 80210, USA}
\ead{ngalatos@du.edu}

\author{George Metcalfe\corref{cor}\fnref{thanks2}}
\address{Mathematical Institute, University of Bern, Sidlerstrasse 5, 3012 Bern, Switzerland}
\ead{george.metcalfe@math.unibe.ch}
\cortext[cor]{Corresponding author}
\fntext[thanks2]{Supported by Swiss National Science Foundation (SNF) grant No. 200021\textunderscore 184693.}	

\author{Simon Santschi}
\address{Mathematical Institute, University of Bern, Sidlerstrasse 5, 3012 Bern, Switzerland}
\ead{simon.santschi@math.unibe.ch}

\begin{abstract}
We prove that an inverse-free equation is valid in the variety $\vty{LG}$ of lattice-ordered groups ($\ell$\nbd{-}groups) if and only if it is valid in the variety $\vty{DLM}$ of distributive lattice-ordered monoids (distributive $\ell$\nbd{-}monoids). This contrasts with the fact that, as proved by Repnitski\u{\i}, there exist inverse-free equations that are valid in all Abelian $\ell$\nbd{-}groups but not in all commutative distributive $\ell$\nbd{-}monoids, and, as we prove here, there exist inverse-free equations that are valid in all totally ordered groups but not in all totally ordered monoids. We also prove that $\vty{DLM}$  has the finite model property and a decidable equational theory, establish a correspondence between the validity of equations in $\vty{DLM}$ and the existence of certain right orders on free monoids, and provide an effective method for reducing the validity of equations in $\vty{LG}$ to the validity of equations in $\vty{DLM}$. 
\end{abstract}

\begin{keyword}
Lattice-ordered groups, distributive lattice-ordered monoids, free groups, free monoids.
\end{keyword}

\end{frontmatter}
 

\section{Introduction}\label{s:introduction}

A {\em lattice-ordered group} ({\em $\ell$\nbd{-}group}) is an algebraic structure $\langle L, \mt, \jn, \cdot, \iv{}, \e \rangle$  such that $\langle L,\cdot, \iv{}, \e \rangle$ is a group, $\langle L, \mt, \jn \rangle$ is a lattice, and the group multiplication preserves the lattice order, i.e., $a \le b$ implies $cad \le cbd$ for all $a,b,c,d \in L$, where  $a\le b\defiff a\mt b = a$. The class of $\ell$\nbd{-}groups forms a variety (equational class) $\vty{LG}$ and admits the following Cayley-style representation theorem:

\begin{theorem}[Holland~\cite{Hol63}]\label{t:holland}
Every $\ell$\nbd{-}group embeds into an $\ell$\nbd{-}group $\m{Aut}(\langle\Om,\le\rangle)$ consisting of the group of order-automorphisms of a totally ordered set (chain) $\langle\Om,\le\rangle$ equipped with the pointwise lattice order. 
\end{theorem}

\noindent
Holland's theorem has provided the foundations for the development of a rich and extensive theory of $\ell$\nbd{-}groups (see~\cite{AF88,KM94} for details and references). In particular, it was proved by Holland~\cite{Hol76} that an equation is valid in $\vty{LG}$ if and only if it is valid in $\m{Aut}(\langle\Q,\le\rangle)$, and by Holland and McCleary~\cite{HM79} that the equational theory of $\vty{LG}$ is decidable.

The inverse-free reduct of any $\ell$\nbd{-}group is a {\em distributive lattice-ordered monoid} ({\em distributive $\ell$\nbd{-}monoid}): an algebraic structure $\langle M, \mt, \jn, \cdot, \e \rangle$ such that $\langle M,\cdot, \e \rangle$ is a monoid, $\langle M, \mt, \jn \rangle$ is a distributive lattice, and the lattice operations distribute over the monoid multiplication, i.e., for all $a,b,c,d\in M$,
\[
a(b\jn c)d = abd \jn acd\quad\text{and}\quad a(b\mt c)d = abd\mt acd.
\]
The class of distributive $\ell$\nbd{-}monoids also forms a variety $\vty{DLM}$ and admits a Cayley-style (or Holland-style) representation theorem:

\begin{theorem}[Anderson and Edwards~\cite{AE84}]\label{t:andersonedwards}
Every distributive $\ell$\nbd{-}monoid embeds into a distributive $\ell$\nbd{-}monoid  $\m{End}(\langle\Om,\le\rangle)$ consisting of the monoid of order-endomorphisms of a chain $\langle\Om,\le\rangle$ equipped with the pointwise lattice order. 
\end{theorem}

Despite the obvious similarity of Theorem~\ref{t:andersonedwards} to Theorem~\ref{t:holland}, the precise nature of the relationship between the varieties of distributive $\ell$\nbd{-}monoids and $\ell$\nbd{-}groups has remained unclear. It was proved by Repnitski\u{\i} in~\cite{Rep83} that the variety of commutative distributive $\ell$\nbd{-}monoids does not have the same equational theory as the class of inverse-free reducts of Abelian $\ell$\nbd{-}groups, but the decidability of its equational theory remains an open problem. In this paper, we prove the following results for the general (noncommutative) case:

\newtheorem*{t:fmp}{Theorem~\ref{t:fmp}}
\begin{t:fmp}
The variety of distributive $\ell$\nbd{-}monoids has the finite model property.\footnote{Recall that a variety $\vty{V}$ has the {\em (strong) finite model property} if an equation (respectively, quasiequation) is valid in $\vty{V}$ if and only if it is valid in the finite members of $\vty{V}$.} More precisely, an equation is valid in all distributive $\ell$\nbd{-}monoids if and only if it is valid in all distributive $\ell$\nbd{-}monoids of order-endomorphisms of a finite chain.
\end{t:fmp}

\newtheorem*{c:decidable}{Corollary~\ref{c:decidable}}
\begin{c:decidable}
The equational theory of distributive $\ell$\nbd{-}monoids is decidable.
\end{c:decidable}

\newtheorem*{t:eqtheory}{Theorem~\ref{t:eqtheory}}
\begin{t:eqtheory}
An inverse-free equation is valid in the variety of $\ell$\nbd{-}groups if and only if it is valid in the variety of distributive $\ell$\nbd{-}monoids. 
\end{t:eqtheory}

Theorem~\ref{t:eqtheory} shows, by way of Birkhoff's variety theorem~\cite{Bir35}, that distributive $\ell$\nbd{-}monoids are precisely the homomorphic images of the inverse-free subreducts of $\ell$\nbd{-}groups. It also allows us, using a characterization of valid $\ell$\nbd{-}group equations given in~\cite{CM19}, to relate the validity of equations in distributive $\ell$\nbd{-}monoids to the existence of certain right orders on free monoids. As a notable consequence of this correspondence, we obtain:

\newtheorem*{c:rofreemonfreegr}{Corollary~\ref{c:rofreemonfreegr}}
\begin{c:rofreemonfreegr}
Every right order on the free monoid over a set $X$ extends to a right order on the free group over $X$.
\end{c:rofreemonfreegr}

To check whether an equation is valid in all distributive $\ell$\nbd{-}monoids, it suffices, by Theorem~\ref{t:eqtheory}, to check the validity of this same equation in all $\ell$\nbd{-}groups. We prove here that a certain converse also holds, namely:

\newtheorem*{t:inversefree}{Theorem~\ref{t:inversefree}}
\begin{t:inversefree}
Let $\eps$ be any $\ell$\nbd{-}group equation with variables in a set $X$. A finite set of inverse-free equations $\Si$ with variables in $X\cup Y$ for some finite set $Y$ can be effectively constructed such that $\eps$ is valid in all $\ell$\nbd{-}groups if and only if the equations in $\Si$ are valid in all distributive $\ell$\nbd{-}monoids.
\end{t:inversefree}

Finally, we turn our attention to totally ordered groups and totally ordered monoids, that is, $\ell$\nbd{-}groups and distributive $\ell$\nbd{-}monoids with a total lattice order. We show that the variety generated by the class of totally ordered monoids can be axiomatized relative to $\vty{DLM}$ by a single equation (Proposition~\ref{p:representabledlm}). However, analogously to the case of commutative distributive $\ell$\nbd{-}monoids and unlike the case of $\vty{DLM}$, we prove:

\newtheorem*{t:subrvarrep}{Theorem~\ref{t:subrvarrep}}
\begin{t:subrvarrep}
There is an inverse-free equation that is valid in all totally ordered groups, but not in all totally ordered monoids.
\end{t:subrvarrep}

We also exhibit an inverse-free equation that is valid in all finite totally ordered monoids, but not in the ordered group of the integers (Proposition~\ref{p:fmpfails}), witnessing the failure of the finite model property for the variety of commutative distributive $\ell$\nbd{-}monoids and the varieties generated by totally ordered monoids and inverse-free reducts of totally ordered groups (Corollary~\ref{c:fmpfails}).


\section{From distributive $\ell$\nbd{-}monoids to $\ell$\nbd{-}groups}\label{s:from_distributive_lmonoids_to_lgroups}

In this section, we establish the finite model property for the variety $\vty{DLM}$ of distributive $\ell$\nbd{-}monoids (Theorem~\ref{t:fmp}) and the decidability of its equational theory (Corollary~\ref{c:decidable}). We then prove that an inverse-free equation is valid in $\vty{DLM}$ if and only if it is valid in the variety $\vty{LG}$ of $\ell$\nbd{-}groups (Theorem~\ref{t:eqtheory}). The key tool for obtaining these results is the notion of a total preorder on a set of monoid terms that is preserved under right multiplication, which bears some similarity to the notion of a diagram employed in~\cite{HM79}. In particular, the existence of such a preorder satisfying a given finite set of inequalities is related to the validity of a corresponding inverse-free equation in $\vty{DLM}$ or $\vty{LG}$.

Let $X$ be any set. We denote by $\m T_{m}(X)$,  $\m T_{g}(X)$,  $\m T_{d}(X)$, and $\m T_{\ell}(X)$ the term algebras over $X$ for monoids, groups, distributive $\ell$\nbd{-}monoids, and $\ell$\nbd{-}groups, respectively, and by $\m F_{m}(X)$,  $\m F_{g}(X)$,  $\m F_{d}(X)$, and $\m F_{\ell}(X)$, the corresponding free algebras, assuming for convenience that $F_{m}(X)\subseteq T_{m}(X)$, $F_{g}(X)\subseteq T_{g}(X)$, $F_{d}(X)\subseteq T_{d}(X)$, and $F_{\ell}(X)\subseteq T_{\ell}(X)$.  Given a set of ordered pairs of monoid terms $S\subseteq F_{m}(X)^2$, we define the set of {\em initial subterms} of $S$:
\[
\Sub(S) := \set{u\in F_m(X) \mid \exists s,t\in F_{m}(X)\colon\langle us,t\rangle\in S \,\text{ or }\,\langle s,ut\rangle\in S}.
\]
Note in particular that $s,t\in\Sub(S)$ for each $\langle s,t\rangle\in S$.

Recall now that a {\em preorder} $\preceq$ on a set $P$ is a binary relation on $P$ that is reflexive and transitive. We write $a\prec b$ to denote that $a\preceq b$ and $b\not\preceq a$, and call $\preceq$ \emph{total} if $a \preceq b$ or $b \preceq a$ for all $a,b\in P$.  
Let $\preceq$ be a preorder on a set of monoid terms $P\subseteq F_{m}(X)$. We say that  $\preceq$ is {\em right-$X$-invariant} if for all $x\in X$, whenever $u\preceq v$ and $ux,vx\in P$, also $ux \preceq vx$, and {\em strictly right-$X$-invariant} if it is right $X$-invariant and for all $x\in X$, whenever  $u\prec v$ and $ux,vx\in P$, also $ux \prec vx$. 

Following standard practice for $\ell$\nbd{-}groups, we write $(p)f$ for the value of a (partial) map $f\colon\Om\to\Om$ defined at $p \in \Om$. As a notational aid, we also often write $\f_r$ to denote the value of a (partial) map $\f$ defined for some element $r$.

\begin{lemma}\label{l:preorderimpliesnotvalid}
Let $S\subseteq F_m(X)^2$ be a finite set of ordered pairs of monoid terms and let $\preceq$ be a total right-$X$-invariant preorder on $\Sub(S)$ satisfying $s\prec t$ for each $\langle s,t\rangle\in S$. 

\begin{enumerate}
\item[\rm (a)] There exists a chain $\langle\Om,\le\rangle$ satisfying $\lvert\Om\rvert \le \lvert\Sub(S)\rvert$, a homomorphism $\f\colon\m{T}_d(X)\to\m{End(\langle\Om,\le\rangle)}$, and some $p\in\Om$ such that $(p)\f_s < (p)\f_t$ for each $\langle s,t\rangle\in S$. 

\item[\rm (b)] If $\preceq$ is also strictly right-$X$-invariant, then there exists a homomorphism $\p\colon\m{T}_\ell(X)\to\m{Aut(\langle\Q,\le\rangle)}$ and some $q\in\Q$ such that $(q)\p_s < (q)\p_t$ for each $\langle s,t\rangle\in S$. 
\end{enumerate}
\end{lemma}
\begin{proof}
For (a), we let $[u] := \set{v\in \Sub(S) \mid u\preceq v \textrm{ and }v\preceq u}$ for each $u\in\Sub(S)$ and define $\Om := \set{[u] \mid u\in \Sub(S)}$, noting that  $\lvert\Om\rvert \le \lvert\Sub(S)\rvert$. If $[u]=[u']$, $[v]=[v']$, and $ u\preceq v$, then  $u'\preceq v'$, so we can define for $[u],[v]\in\Om$,
\[
[u]\leq [v]\defiff u\preceq v.
\]  
Clearly, $\leq$ is a total order on $\Om$ and $[s] < [t]$ for each $\langle s,t\rangle\in S$. Moreover, if  $[u],[v]\in\Om$, $x\in X$, and $ux,vx \in \Sub(S)$, then, using the right-$X$-invariance of $\preceq$, 
\[
[u]\le[v] \:\Longrightarrow\: [ux]\le[vx].
\]
In particular, if $[u]=[v]\in\Om$, $x\in X$, and $ux,vx \in \Sub(S)$, then $[ux]=[vx]$. Hence for each $x\in X$, we obtain a partial order-endomorphism $\tilde{\f}_x\colon\Om\to\Om$ of $\langle\Om,\le\rangle$ by defining $([u])\tilde{\f}_x:=[ux]$ whenever $[u]\in\Om$ and $ux \in \Sub(S)$. Moreover, each of these partial maps $\tilde{\f}_x$ extends to an order-endomorphism $\f_x\colon\Om\to\Om$  of $\langle\Om,\le\rangle$. Now let  $\f\colon\m{T}_d(X)\to\m{End(\langle\Om,\le\rangle)}$ be the homomorphism extending the assignment $x\mapsto \f_x$. Then $([\e])\f_u = [u]$ for every $u\in \Sub(S)$ and hence $([\e])\f_s < ([\e])\f_t$ for each $\langle s,t\rangle\in S$.

For (b), note that the set $\Om$ defined in (a) is finite and, assuming that $\preceq$ is strictly right-$X$-invariant, the partial order-endomorphisms $\tilde{\f}_x\colon\Om\to\Om$ of $\langle\Om,\le\rangle$ for $x\in X$ are injective. Hence $\langle\Om,\le\rangle$ can be identified with a subchain of $\langle\Q,\le\rangle$ and each $\tilde{\f}_x$ can be extended to an order-automorphism $\p_x\colon\Q\to\Q$ of $\langle\Q,\le\rangle$. As in (a), we obtain a homomorphism $\p\colon\m{T}_\ell(X)\to\m{Aut(\langle\Q,\le\rangle)}$  extending the assignment $x\mapsto \p_x$ such that $([\e])\f_s < ([\e])\f_t$ for each $\langle s,t\rangle\in S$.
\end{proof}

For $s,t\in T_{\ell}(X)$, we write $s \le t$ as an abbreviation for the equation $s\mt{t}\eq{s}$, noting that $s\eq{t}$ is valid in an $\ell$\nbd{-}monoid or $\ell$\nbd{-}group $\m{L}$ if and only if  $s\le{t}$ and $t\le{s}$ are valid in $\m{L}$.  It is easily seen that every $\ell$\nbd{-}group (or $\ell$\nbd{-}monoid) term is equivalent in $\vty{LG}$ (or $\vty{DLM}$) to both a join of meets of group (monoid) terms and a meet of joins of group (monoid) terms. It follows that to check the validity of an (inverse-free) equation in $\vty{LG}$ (or $\vty{DLM}$), it suffices to consider equations of the form $\bigwedge_{i=1}^n t_{i}\leq \bigvee_{j=1}^m s_{j}$ where $s_j,t_i\in{F_{g}(X)}$ (or  $s_j,t_i\in{F_{m}(X)}$) for $1\le{i}\le{n}$,  $1\le{j}\le{m}$. The next lemma relates the validity of an inverse-free equation of this form in $\vty{LG}$ or $\vty{DLM}$ to the existence of a total (strictly) right-$X$-invariant preorder on a corresponding set of initial subterms.

\begin{lemma}\label{l:notvalidimpliespreorder}
Let $\eps=(\bigwedge_{i=1}^n t_{i}\leq \bigvee_{j=1}^m s_{j})$ where $s_j,t_i\in{F_{m}(X)}$ for $1\le{i}\le{n}$,  $1\le{j}\le{m}$, and let $S := \{\langle s_j, t_i\rangle\in F_{m}(X)^2\mid 1\le i\le n,\,1\le j\le m\}$.
\begin{enumerate}
\item[\rm (a)] $\vty{DLM}\models\eps$ if and only if  there is no total right-$X$-invariant preorder $\preceq$ on $\Sub(S)$ satisfying $s\prec t$ for each $\langle s,t\rangle\in S$.
\item[\rm (b)] $\vty{LG}\models\eps$ if and only if there is no total strictly right-$X$-invariant preorder $\preceq$ on $\Sub(S)$ satisfying $s\prec t$ for each $\langle s,t\rangle\in S$.
\end{enumerate}
\end{lemma}
\begin{proof}
For the left-to-right direction of (a), suppose contrapositively that there exists a total right-$X$-invariant preorder $\preceq$ on $\Sub(S)$ satisfying $s\prec t$ for each $\langle s,t\rangle\in S$. By Lemma~\ref{l:preorderimpliesnotvalid}(a), there exist a chain $\langle\Om,\le\rangle$, a homomorphism $\f\colon\m{T}_d(X)\to\m{End(\langle\Om,\le\rangle)}$, and some $p\in\Om$ such that $(p)\f_s < (p)\f_t$ for each $\langle s,t\rangle\in S$. So $(p)\f_{\bigwedge_{i=1}^n t_{i}}>(p)\f_{\bigvee_{j=1}^m s_{j}}$, and hence $\vty{DLM}\not\models\eps$. Similarly, for the left-to-right direction of (b), there exist, by  Lemma~\ref{l:preorderimpliesnotvalid}(b), a homomorphism $\p\colon\m{T}_\ell(X)\to\m{Aut(\langle\Q,\le\rangle)}$ and some $q\in\Q$ such that $(q)\p_{\bigwedge_{i=1}^n t_{i}}>(q)\p_{\bigvee_{j=1}^m s_{j}}$ and hence $\vty{LG}\not\models\eps$. 

For the right-to-left direction of (a), suppose contrapositively that $\vty{DLM}\not\models\eps$. By Theorem~\ref{t:andersonedwards}, there exist a chain $\langle\Om,\le\rangle$, a homomorphism $\f\colon\m{T}_d(X)\to\m{End(\langle\Om,\le\rangle)}$, and some $p\in\Om$ such that $\bigwedge_{i=1}^n (p)\f_{t_i} > \bigvee_{j=1}^m (p)\f_{s_j}$. Then $(p)\f_t >(p)\f_s$ for each $\langle s,t\rangle\in S$ and we define for $u,v\in\Sub(S)$,
\[
u \preceq v\defiff (p)\f_u \leq (p)\f_v.
\]
Clearly $\preceq$ is a total preorder satisfying $s\prec t$ for each $\langle s,t\rangle\in S$. Moreover, since $\f$ is a homomorphism, $\preceq$ is right-$X$-invariant on $\Sub(S)$.

For the right-to-left direction of (b), suppose that $\vty{LG}\not\models\eps$. By Theorem~\ref{t:holland}, there exist a chain $\langle\Om,\le\rangle$, a homomorphism  $\p\colon\m{T}_\ell(X)\to\m{Aut(\langle\Om,\le\rangle)}$, and $q\in\Om$ such that $\bigwedge_{i=1}^n (q)\p_{t_i} >\bigvee_{j=1}^m (q)\p_{s_j}$. The proof then proceeds exactly as in the case of (a), except that we may observe finally that $\preceq$ is strictly right-$X$-invariant on $\Sub(S)$, using the fact that $\p_u$ is bijective for each $u\in\Sub(S)$.
\end{proof}

We now combine the first parts of the preceding lemmas to obtain:

\begin{theorem}\label{t:fmp}
The variety of distributive $\ell$\nbd{-}monoids has the finite model property. More precisely, an equation is valid in all distributive $\ell$\nbd{-}monoids if and only if it is valid in all distributive $\ell$\nbd{-}monoids of order-endomorphisms of a finite chain.
\end{theorem}
\begin{proof}
It suffices to establish the result for an equation $\eps=(\bigwedge_{i=1}^n t_{i}\leq \bigvee_{j=1}^m s_{j})$, where $s_1,\dots,s_m,t_1,\dots,t_n\in F_{m}(X)$. Suppose that $\vty{DLM}\not\models\eps$ and let $S := \{\langle s_j, t_i\rangle\mid 1\le i\le n,\,1\le j\le m\}$. Combining Lemmas~\ref{l:notvalidimpliespreorder}(a) and~\ref{l:preorderimpliesnotvalid}(a), there exist a finite chain $\langle\Om,\le\rangle$, a homomorphism $\f\colon\m{T}_d(X)\to\m{End(\langle\Om,\le\rangle)}$, and some $p\in\Om$ such that $(p)\f_s <(p)\f_t$ for each $\langle s,t\rangle\in S$. But then $(p)\f_{\bigwedge_{i=1}^n t_{i}}>(p)\f_{\bigvee_{j=1}^m s_{j}}$, so  $\m{End(\langle\Om,\le\rangle)}\not\models\eps$. 
\end{proof}

Since $\vty{DLM}$ is a finitely axiomatized variety, we also obtain:

\begin{corollary}\label{c:decidable}
The equational theory of distributive $\ell$\nbd{-}monoids is decidable.
\end{corollary}

Similarly, the second parts of Lemmas~\ref{l:preorderimpliesnotvalid} and~\ref{l:notvalidimpliespreorder} can be used to show that an inverse-free equation is valid in all $\ell$\nbd{-}groups if and only if it is valid in $\m{Aut(\langle\Q,\le\rangle)}$. Indeed, this correspondence is known to hold for all equations.

\begin{theorem}[\cite{Hol76}]\label{t:AutQgeneratesLG}
An equation is valid in all $\ell$\nbd{-}groups if and only if it is valid in $\m{Aut(\langle\Q,\le\rangle)}$.
\end{theorem}

Lemma~\ref{l:fromendtoaut} below provides the key ingredient for showing that an inverse-free equation is valid in $\vty{LG}$ if and only if it is valid in $\vty{DLM}$. First, we illustrate the rather involved construction in the proof of this lemma with a simple example.

\begin{example}\label{ex:endtoaut}
Let $\m{End}(\m{2})$ be the distributive $\ell$\nbd{-}monoid of order-endomorphisms of the two-element chain $\m{2}=\langle\{0,1\},\le\rangle$, and let $\langle k_0,k_1 \rangle$ denote the member of $\m{End}(\m{2})$ with $0 \mapsto k_0$ and $1 \mapsto k_1$. The equation $yxy \leq xyx$ fails in $\m{End}(\m{2})$, since for the homomorphism $\f \colon \m{T}_d(\{x,y\}) \to \m{End}(\m{2})$ extending the assignment $x \mapsto \f_x = \langle 0,0\rangle$ and  $y \mapsto \f_y = \langle 1,1\rangle$, we obtain
\[
(1)\f_{yxy}=(((1)\f_y)\f_x)\f_y = 1> 0= (((1)\f_x)\f_y)\f_x =(1)\f_{xyx}.
\]
Let $S:=\{\langle xyx, yxy\rangle\}$. Then $\f$ yields a total right-$\{x,y\}$-invariant preorder $\preceq$ on $\Sub(S)=\{\e,x,y,xy,yx,xyx,yxy\}$ given by $x \sim yx \sim xyx \prec \e \sim  y \sim xy \sim yxy$, since $(1)\f_{x}=(1)\f_{yx}=(1)\f_{xyx}=0 < 1=(1)\f_{\e}=(1)\f_{y}=(1)\f_{xy}=(1)\f_{yxy}$. Note that $\preceq$ is not strictly right-$\{x,y\}$-invariant, since $x \prec \e$, but $xy \sim y$; this corresponds to the fact that $\f_y$ is not a partial bijective map on $\{0,1\}$, as $0 <1$ and $(0)\f_y = (1)\f_y$.

We describe a total strictly right-$\{x,y\}$-invariant preorder $\trianglelefteq$ on $\Sub(S)$ such that ${\prec}\subseteq{\vartriangleleft}$. This corresponds to constructing partial bijections $\widehat{\f}_x$ and $\widehat{\f}_y$ on $\Sub(S)$ that extend $\f_x$ and $\f_y$, respectively. The relation $\vartriangleleft$ can be computed directly using the definition given in Lemma~\ref{l:fromendtoaut}, but to provide both a simpler description and intuition for the construction, we identify each element  $x_k \cdots x_1$ of  $\Sub(S)$ with the sequence $((1)\f_\e,(1)\f_{x_k}, \ldots, (1)\f_{x_k \cdots x_1})$, so
$\e=(1)$, $x=(1,0)$, $y=(1,1)$, $xy=(1,0,1)$, $yx=(1,1,0)$, $xyx=(1,0,1,0)$, and $yxy=(1,1,0,1)$. Note that these are the paths of elements of $\{0,1\}$ involved in the successive computation steps for each term at the point $p=1$ and  can be visualized as indicated in Figure~\ref{fig:onedlm}.
          
\begin{figure}[H]
\begin{center}
\begin{tikzpicture}[line cap=round,line join=round,>=triangle 45,x=1.0cm,y=1.0cm]
\tikzset{myarrow/.style={->, >=latex', shorten >=1pt, thick}, myarrowone/.style={-, >=latex', shorten >=1pt, dotted, thin}}
\draw[-,color=black] (0.,0.5) -- (0.,1.);
\draw[-,color=black] (1.5,0.5) -- (1.5,1);
\draw[-,color=black] (3,0.5) -- (3,1);
\draw[-,color=black] (4.5,0.5) -- (4.5,1);
\begin{scriptsize}
\draw[color=black] (-0.175,0.5) node {$0$};
\draw[color=black] (-0.175,1.0) node {$\bf{1}$};
\draw[color=black] (4.675,0.5) node {$\bf{0}$};
\draw[color=black] (4.675,1.0) node {$1$};
\draw [fill=black] (4.5,1.) circle (0.5pt);
\draw [fill=black] (3.,1.) circle (0.5pt);
\draw [fill=black] (4.5,0.5) circle (0.5pt);
\draw [fill=black] (1.5,0.5) circle (0.5pt);
\draw [fill=black] (0.,1.) circle (0.5pt);
\draw [fill=black] (0.,0.5) circle (0.5pt);
\draw [fill=black] (3.,0.5) circle (0.5pt);
\draw [fill=black] (1.5,1.) circle (0.5pt);
\draw[color=black] (0.75,1.25) node {$\f_x$};
\draw[color=black] (2.25,1.25) node {$\f_y$};
\draw[color=black] (3.75,1.25) node {$\f_x$};
\draw[myarrow] (0,1) -- (1.5,0.5);
\draw[myarrow, gray] (0.,0.5) -- (1.5,0.5);
\draw[myarrow, gray] (1.5,1) -- (3.,1);
\draw[myarrow] (1.5,0.5) -- (3.,1);
\draw[myarrow] (3,1) -- (4.5,0.5);
\draw[myarrow, gray] (3,0.5) -- (4.5,0.5);
\end{scriptsize}
\end{tikzpicture}
\hspace*{0.5cm}
\begin{tikzpicture}[line cap=round,line join=round,>=triangle 45,x=1.0cm,y=1.0cm]
\tikzset{myarrow/.style={->, >=latex', shorten >=1pt, thick}, myarrowone/.style={-, >=latex', shorten >=1pt, dotted, thin}}
\draw[-,color=black] (0.,0.5) -- (0.,1.);
\draw[-,color=black] (1.5,0.5) -- (1.5,1);
\draw[-,color=black] (3,0.5) -- (3,1);
\draw[-,color=black] (4.5,0.5) -- (4.5,1);
\begin{scriptsize}
\draw[color=black] (-0.175,0.5) node {$0$};
\draw[color=black] (-0.175,1.0) node {$\bf{1}$};
\draw[color=black] (4.675,0.5) node {$0$};
\draw[color=black] (4.675,1.0) node {$\bf{1}$};
\draw [fill=black] (4.5,1.) circle (0.5pt);
\draw [fill=black] (3.,1.) circle (0.5pt);
\draw [fill=black] (4.5,0.5) circle (0.5pt);
\draw [fill=black] (1.5,0.5) circle (0.5pt);
\draw [fill=black] (0.,1.) circle (0.5pt);
\draw [fill=black] (0.,0.5) circle (0.5pt);
\draw [fill=black] (3.,0.5) circle (0.5pt);
\draw [fill=black] (1.5,1.) circle (0.5pt);
\draw[color=black] (0.75,1.25) node {$\f_y$};
\draw[color=black] (2.25,1.25) node {$\f_x$};
\draw[color=black] (3.75,1.25) node {$\f_y$};
\draw[myarrow] (0,1) -- (1.5,1);
\draw[myarrow, gray] (0.,0.5) -- (1.5,1);
\draw[myarrow] (1.5,1) -- (3.,0.5);
\draw[myarrow, gray] (1.5,0.5) -- (3.,0.5);
\draw[myarrow, gray] (3,1) -- (4.5,1);
\draw[myarrow] (3,0.5) -- (4.5,1);
\end{scriptsize}
\end{tikzpicture}
\end{center}
\caption{The paths for $xyx=(1,0,1,0)$ and $yxy=(1,1,0,1)$.}\label{fig:onedlm}
\end{figure}
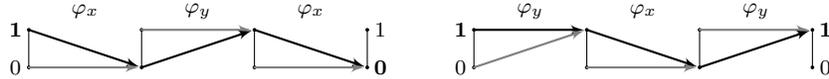
The relation $\vartriangleleft$ on these paths is simply the reverse lexicographic order:
\[
(1,0)\vartriangleleft(1,0,1,0)\vartriangleleft(1,1,0)\vartriangleleft(1)\vartriangleleft(1,0,1)\vartriangleleft(1,1,0,1)
\vartriangleleft(1,1),
\]
where the first three elements serve as copies of $0$ and the last four as copies of $1$, so via the above identification we obtain
\[
x\vartriangleleft xyx\vartriangleleft yx\vartriangleleft \e\vartriangleleft xy\vartriangleleft yxy\vartriangleleft y.
\]
It can be verified that this is a total strictly right-$\{x,y\}$-invariant (pre)order, or, more easily, that the corresponding partial order-endomorphisms $\widehat{\f}_x$ and $\widehat{\f}_y$ are partial bijections as shown in Figure~\ref{fig:twodlm}.

\begin{figure}[H]
\begin{center}
\begin{tikzpicture}[line cap=round,line join=round,>=triangle 45,x=1.0cm,y=1.0cm]
\tikzset{myarrow/.style={->, >=latex', shorten >=1pt, thick}, myarrowone/.style={-, >=latex', shorten >=1pt, dotted, thin}}
\draw[-,color=black] (0.,2.) -- (0.,2.5);
\draw[-,color=black] (0.,3.) -- (0.,3.5);
\draw[-,color=black] (0.,4.) -- (0.,4.5);
\draw[-,color=black] (0.,2.5) -- (0.,3.);
\draw[-,color=black] (0.,3.5) -- (0.,4.);
\draw[-,color=black] (0.,4.5) -- (0.,5.);

\draw[-,color=black] (1.5,2.) -- (1.5,2.5);
\draw[-,color=black] (1.5,3.) -- (1.5,3.5);
\draw[-,color=black] (1.5,4.) -- (1.5,4.5);
\draw[-,color=black] (1.5,2.5) -- (1.5,3.);
\draw[-,color=black] (1.5,3.5) -- (1.5,4.);
\draw[-,color=black] (1.5,4.5) -- (1.5,5.);

\draw[-,color=black] (3.,2.) -- (3.,2.5);
\draw[-,color=black] (3.,3.) -- (3.,3.5);
\draw[-,color=black] (3.,4.) -- (3.,4.5);
\draw[-,color=black] (3.,2.5) -- (3.,3.);
\draw[-,color=black] (3.,3.5) -- (3.,4.);
\draw[-,color=black] (3.,4.5) -- (3.,5.);

\draw[-,color=black] (4.5,2.) -- (4.5,2.5);
\draw[-,color=black] (4.5,3.) -- (4.5,3.5);
\draw[-,color=black] (4.5,4.) -- (4.5,4.5);
\draw[-,color=black] (4.5,2.5) -- (4.5,3.);
\draw[-,color=black] (4.5,3.5) -- (4.5,4.);
\draw[-,color=black] (4.5,4.5) -- (4.5,5.);

\begin{scriptsize}
\draw[color=black] (-0.4,5)    node {$(1,1)\,\,\,\,$};
\draw[color=black] (-0.6,4.5) node {$(1,1,0,1)\,\,\,\,\,\,$};
\draw[color=black] (-0.5,4.0) node {$(1,0,1)\,\,\,\,\,$};
\draw[color=black] (-0.3,3.5) node {$(1)\,\,$};
\draw[color=black] (-0.5,3.0) node {$(1,1,0)\,\,\,\,\,$};
\draw[color=black] (-0.6,2.5) node {$(1,0,1,0)\,\,\,\,\,\,$};
\draw[color=black] (-0.4,2.0) node {$(1,0)\,\,\,\,$};

\draw[color=black] (5.2,5)    node {$(1,1)=y$};
\draw[color=black] (5.65,4.5) node {$(1,1,0,1)=yxy$};
\draw[color=black] (5.45,4.0) node {$(1,0,1)=xy$};
\draw[color=black] (5.05,3.5) node {$(1)=\e$};
\draw[color=black] (5.45,3.0) node {$(1,1,0)=yx$};
\draw[color=black] (5.65,2.5) node {$(1,0,1,0)=xyx$};
\draw[color=black] (5.2,2.0) node {$(1,0)=x$};

\draw [fill=black] (0.,5.) circle (0.5pt);
\draw [fill=black] (0.,4.) circle (0.5pt);
\draw [fill=black] (0.,4.5) circle (0.5pt);
\draw [fill=black] (0.,3.) circle (0.5pt);
\draw [fill=black] (0.,2.) circle (0.5pt);
\draw [fill=black] (0.,3.5) circle (0.5pt);
\draw [fill=black] (0.,2.5) circle (0.5pt);

\draw [fill=black] (1.5,5.) circle (0.5pt);
\draw [fill=black] (1.5,4.) circle (0.5pt);
\draw [fill=black] (1.5,4.5) circle (0.5pt);
\draw [fill=black] (1.5,3.) circle (0.5pt);
\draw [fill=black] (1.5,2.) circle (0.5pt);
\draw [fill=black] (1.5,3.5) circle (0.5pt);
\draw [fill=black] (1.5,2.5) circle (0.5pt);

\draw [fill=black] (3.,5.) circle (0.5pt);
\draw [fill=black] (3.,4.) circle (0.5pt);
\draw [fill=black] (3.,4.5) circle (0.5pt);
\draw [fill=black] (3.,3.) circle (0.5pt);
\draw [fill=black] (3.,2.) circle (0.5pt);
\draw [fill=black] (3.,3.5) circle (0.5pt);
\draw [fill=black] (3.,2.5) circle (0.5pt);

\draw [fill=black] (4.5,5.) circle (0.5pt);
\draw [fill=black] (4.5,4.) circle (0.5pt);
\draw [fill=black] (4.5,4.5) circle (0.5pt);
\draw [fill=black] (4.5,3.) circle (0.5pt);
\draw [fill=black] (4.5,2.) circle (0.5pt);
\draw [fill=black] (4.5,3.5) circle (0.5pt);
\draw [fill=black] (4.5,2.5) circle (0.5pt);

\draw[color=black] (0.75,5.25) node {$\widehat{\f}_{y}$};
\draw[color=black] (2.25,5.25) node {$\widehat{\f}_x$};
\draw[color=black] (3.75,5.25) node {$\widehat{\f}_{y}$};
\draw[myarrow] (0.,3.5) -- (1.5,5);
\draw[myarrow, gray] (0.,3) -- (1.5,4.5);
\draw[myarrow, gray] (0.,2) -- (1.5,4);

\draw[myarrow] (1.5,5) -- (3.,3);
\draw[myarrow, gray] (1.5,4) -- (3.,2.5);
\draw[myarrow, gray] (1.5,3.5) -- (3.,2);

\draw[myarrow, gray] (3,3.5) -- (4.5,5);
\draw[myarrow] (3,3) -- (4.5,4.5);
\draw[myarrow, gray] (3,2) -- (4.5,4);

\end{scriptsize}
\end{tikzpicture}
\caption{The partial bijections $\widehat{\f}_x$ and $\widehat{\f}_y$ and the evaluation of $\widehat{\f}_{xyx}$ at $(1)=\e$.}\label{fig:twodlm}
\end{center}
\end{figure}
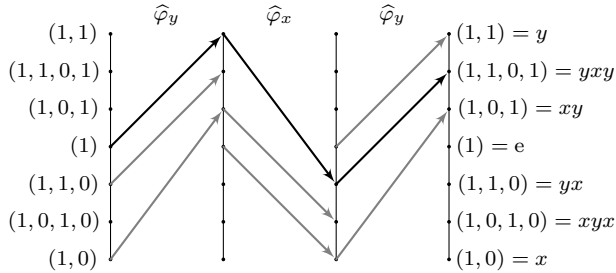
\end{example}

\begin{lemma}\label{l:fromendtoaut} 
For any $S\subseteq F_m(X)^2$ and total right-X-invariant preorder $\preceq$ on $\Sub(S)$, there exists a total strictly right-X-invariant preorder $\trianglelefteq$ on $\Sub(S)$ such that ${\prec}\subseteq{\vartriangleleft}$. 
\end{lemma}
\begin{proof} 
We define the following relations on $\Sub(S)$:
\begin{align*}
u\sim v\defiff 
&\: u\preceq v\,\text{ and }\,v\preceq u;\\[.05in]
x_k\cdots x_1\vartriangleleft y_l \cdots y_1\defiff 
&\: \exists j\leq l+1\colon x_k\cdots x_i \sim y_l \cdots y_i \text{ for all } i<j \text{ and }\\     
&\:  (x_k\cdots x_j \prec y_l \cdots y_j\,\text{ or }\, j=k+2);\\[.05in]
x_k\cdots x_1\equiv y_l \cdots y_1\defiff 
&\: k=l \,\text{ and }\, x_k\cdots x_i \sim y_l \cdots y_i \text{ for each } i \le k;\\[.05in]
u\trianglelefteq v\defiff 
&\: u\vartriangleleft v \,\text{ or }\, u\equiv v,
\end{align*}
assuming that $x_k\cdots x_i$ is the empty product $\e$ for $i>k$. 

Observe that setting $j=1$ in the definition of $\vartriangleleft$ yields ${\prec}\subseteq{\vartriangleleft}$. Also $u\vartriangleleft v$ implies $u \not\equiv v$. The irreflexivity of $\vartriangleleft$ follows directly from the fact that $\prec$ is irreflexive. For the transitivity of $\vartriangleleft$, we consider $u,v,w \in \Sub(S)$ satisfying $u=x_k\cdots x_1$, $v = y_l\cdots y_1$, $w=z_m\cdots z_1$,  $u\vartriangleleft v$, and $v\vartriangleleft w$. By definition, there exists a $j_1\leq l+1$ such that $x_k\cdots x_i \sim y_l \cdots y_i$ for all $i<j_1$,  and either  $x_k\cdots x_{j_1} \prec y_l \cdots y_{j_1}$ or $j_1=k+2 $, and there exists a $j_2\leq m+1$ such that $y_l\cdots y_i \sim z_m\cdots z_i$ for all  $i<j_2$, and either $y_l\cdots y_{j_2} \prec z_{m} \cdots z_{j_2}$ or $j_2 = l+2 $. There are four cases to check:

\begin{enumerate}

\item $x_k\cdots x_{j_1} \prec y_l \cdots y_{j_1}$ and $y_l\cdots y_{j_2} \prec z_{m} \cdots z_{j_2}$. If  $j_2 \leq j_1$, then $x_k\cdots x_i \sim y_l\cdots y_i \sim z_m\cdots z_i$  for all $i<j_2$ and $x_k\cdots x_{j_2} \sim y_l\cdots y_{j_2} \prec z_m \cdots z_{j_2}$, so (since $\sim$ and $\preceq$ are transitive), $u\vartriangleleft w$. If  $j_1 < j_2$, then $j_1 \leq m+1$ and $x_k\cdots x_i \sim y_l\cdots y_i \sim z_m\cdots z_i$  for all $i<j_1$ and $x_k\cdots x_{j_1} \prec y_l\cdots y_{j_1} \sim z_m \cdots z_{j_1}$, so  $u\vartriangleleft w$.

\item $x_k\cdots x_{j_1} \prec y_l \cdots y_{j_1}$ and $j_2 = l+2$. Then $j_1 \leq l+1 < j_2 \leq m+1$, so $x_k\cdots x_i \sim y_l\cdots y_i \sim z_m\cdots z_i$  for all $i<j_1$ and $x_k\cdots x_{j_1} \prec y_l\cdots y_{j_1} \sim z_m \cdots z_{j_1}$. Hence $u\vartriangleleft w$.

\item $j_1 = k+2$ and $y_l\cdots y_{j_2} \prec z_{m} \cdots z_{j_2}$. If $j_2 < j_1$, then $x_k\cdots x_i \sim y_l\cdots y_i \sim z_m\cdots z_i$  for all $i<j_2$ and $x_k\cdots x_{j_2} \sim y_l\cdots y_{j_2} \prec z_m \cdots z_{j_2}$, so $u\vartriangleleft w$. If $j_1 \leq j_2$, then $j_1 \leq m+1$, $x_k\cdots x_i \sim y_l\cdots y_i \sim z_m\cdots z_i$ for all $i<j_1$, and $j_1 = k+2$, so $u\vartriangleleft w$.

\item $j_1 = k+2$ and $j_2 = l+2$. Then $j_1 \leq m+1$ and $x_k\cdots x_i \sim y_l\cdots y_i \sim z_m\cdots z_i$  for all $i<j_1$. Hence $u\vartriangleleft w$.
\end{enumerate}

For the transitivity of $\trianglelefteq$, there are also several cases to check. Clearly, if $u\vartriangleleft v$ and $v\vartriangleleft w$, then $u\vartriangleleft w$, by the transitivity of $\vartriangleleft$. If $u\vartriangleleft v$ and $v\equiv w$, then $u\vartriangleleft w$, using the definition of $\vartriangleleft$ and $\equiv$ and the transitivity of $\sim$ and $\prec$. Similarly, if $u\equiv v$ and $v\vartriangleleft w$, then $u\vartriangleleft w$. Finally, if $u\equiv v$ and $v\equiv w$, then $u\equiv w$, by the transitivity of $\sim$. Moreover, $\trianglelefteq$ is reflexive, since $u\equiv u$ for any $u\in \Sub(S)$, so  $\trianglelefteq$ is a preorder.  Since  $\preceq$ is total, $u \not\vartriangleleft v$ and $v  \not\vartriangleleft u$ implies  $u\equiv v$; so $\trianglelefteq$ is total. Note also that $u\vartriangleleft v$ if and only if $u\trianglelefteq v$ and $v \not\trianglelefteq u$ as suggested by the notation.
 
 To prove that $\trianglelefteq$ is strictly right-$X$-invariant on $\Sub(S)$, consider $x\in X$ and $u,v\in \Sub(S)$ such that  $u\trianglelefteq v$ and $ux,vx\in \Sub(S)$. Suppose first that $u\equiv v$, so $u$ and $v$ have the same length and $u\sim v$. Then $ux$ and $vx$ have the same length and, since $\preceq$ is right-$X$-invariant, $ux\sim vx$. So $ux\equiv vx$ and hence $ux\trianglelefteq vx$. Now suppose that $u\vartriangleleft v$.  If $ux \prec vx$, then $ux\vartriangleleft vx$. Also, if  $ux \sim vx$, then, since $u\vartriangleleft v$, the definition of $\vartriangleleft$ gives  $ux\vartriangleleft vx$. Finally,  suppose towards a contradiction that $ux\not\preceq vx$. Since $\preceq$ is right-$X$-invariant, $u\not\preceq v$. But then, since $\preceq$ is total, $v \prec u$ and so $v\vartriangleleft u$, contradicting  $u\vartriangleleft v$.
\end{proof}

\begin{proposition}\label{p:eqtheory}
An inverse-free equation is valid in all distributive $\ell$\nbd{-}monoids if and only if it is valid in $\m{Aut(\langle\Q,\le\rangle)}$.
\end{proposition}
\begin{proof}
The left-to-right direction follows directly from the fact that the inverse-free reduct of $\m{Aut(\langle\Q,\le\rangle)}$ is a distributive $\ell$\nbd{-}monoid. For the converse, suppose without loss of generality that $\vty{DLM}\not\models\bigwedge_{i=1}^n t_{i}\leq \bigvee_{j=1}^m s_{j}$, where $s_j,t_i\in{F_{m}(X)}$ for $1\le{i}\le{n}$,  $1\le{j}\le{m}$, and let $S := \{\langle s_j, t_i\rangle\mid 1\le i\le n,\,1\le j\le m\}$. By Lemma~\ref{l:notvalidimpliespreorder}(a), there exists a total right-$X$-invariant preorder $\preceq$ on $\Sub(S)$ satisfying $s\prec t$ for each $\langle s,t\rangle\in S$. By Lemma~\ref{l:fromendtoaut}, there exists a total strictly right-X-invariant preorder $\trianglelefteq$ on $\Sub(S)$ such that ${\prec}\subseteq{\vartriangleleft}$. In particular, $s\vartriangleleft t$ for each $\langle s,t\rangle \in S$. Hence, by Lemma~\ref{l:preorderimpliesnotvalid}(b), there exist a homomorphism $\p\colon\m{T}_\ell(X)\to\m{Aut(\langle\Q,\le\rangle)}$ and $q\in\Q$ such that $(q)\p_{s_j} <(q)\p_{t_i}$ for $1\le i\le n$, $1\le j\le m$. So $\m{Aut(\langle\Q,\le\rangle)}\not\models\bigwedge_{i=1}^n t_{i}\leq \bigvee_{j=1}^m s_{j}$. 
\end{proof}

The main result of this section now follows directly from Proposition~\ref{p:eqtheory} and the fact that the inverse-free reduct of any $\ell$\nbd{-}group is a distributive $\ell$\nbd{-}monoid.

\begin{theorem}\label{t:eqtheory}
An inverse-free equation is valid in the variety of $\ell$\nbd{-}groups if and only if it is valid in the variety of distributive $\ell$\nbd{-}monoids. 
\end{theorem}

\noindent
It follows by Birkhoff's variety theorem~\cite{Bir35} that $\vty{DLM}$ is generated as a variety by the class of inverse-free reducts of $\ell$\nbd{-}groups and hence that distributive $\ell$\nbd{-}monoids are precisely the homomorphic images of the inverse-free subreducts of $\ell$\nbd{-}groups.

Since the equational theories of the varieties of distributive lattices~\cite{HRB87} and $\ell$\nbd{-}groups~\cite{GM16} are co-NP-complete, we also obtain the following complexity result:

\begin{corollary}
The equational theory of distributive $\ell$\nbd{-}monoids is co-NP-complete.
\end{corollary}

The correspondence between $\ell$\nbd{-}groups and distributive $\ell$\nbd{-}monoids established in Theorem~\ref{t:eqtheory} does not extend to inverse-free quasiequations. In particular, the quasiequation $xz\eq yz\,\Longrightarrow\,x\eq y$, describing right cancellativity, is valid in all $\ell$\nbd{-}groups, but not in the distributive $\ell$\nbd{-}monoid $\m{End}(\m{2})$. A further example is the quasiequation $xy\eq\e\,\Longrightarrow\,yx\eq\e$, which is clearly valid in all $\ell$\nbd{-}groups, but not in the distributive $\ell$\nbd{-}monoid $\m{End}(\langle\N,\le\rangle)$. To see this, define $f,g\in {\rm End}(\langle\N,\le\rangle)$ by $(n)f:=n+1$ and $(n)g:=\max(n-1,0)$; then $(n)fg=n$ for all $n\in\N$, but $(0)gf=1$. Let us also remark, however, that this quasiequation is valid in any finite distributive $\ell$\nbd{-}monoid $\m{L}$. If $ab=\e$ for some $a,b\in L$, then, by finiteness, $a^n=a^{n+k}$ for some $n,k\in\N^{>0}$, so $\e=a^nb^n= a^{n+k}b^n=a^k$ and $ba = a^k ba = a^{k-1}aba = a^k=\e$. Hence the variety of distributive $\ell$\nbd{-}monoids does not have the strong finite model property.


\section{Right orders on free groups and free monoids}\label{s:right_orders_on_free_groups_and_free_monoids}

In this section, we use Theorem~\ref{t:eqtheory} and a characterization of valid $\ell$\nbd{-}group equations in $\vty{LG}$ given in~\cite{CM19} to relate the existence of a right order on a free monoid satisfying some finite set of inequalities to the validity of an equation in $\vty{DLM}$ (Theorem~\ref{p:rovalidmon}). In particular, it follows that any right order on the free monoid over a set $X$ extends to a right order on the free group over $X$ (Corollary~\ref{c:rofreemonfreegr}).

Recall first that a {\em right order} on a monoid (or group) $\m{M}$ is a total order $\le$ on $M$ such that $a \le b$ implies $ac \le bc$ for any $a,b,c\in{M}$; in this case, $\m{M}$ is said to be {\em right-orderable}. Left orders and left-orderability are defined symmetrically.

The following result of~\cite{CM19} establishes a correspondence between the validity of an equation in $\vty{LG}$ and the existence of a right order on a free group with a negative cone (or, by duality, a positive cone) containing certain elements.

\begin{theorem}[{\cite[Theorem~2]{CM19}}]\label{t:freegroupsandrightorders}
Let $s_1,\dots,s_m\in F_g(X)$. Then $\vty{LG}\models\e \le  \bigvee_{j=1}^m s_{j}$ if and only if there is no right order $\le$ on $\m{F}_g(X)$ satisfying $s_j<\e$ for $1\le j\le m$.
\end{theorem}

\noindent
Combining this result with Theorem~\ref{t:eqtheory}, we obtain a correspondence between the validity of an equation in $\vty{DLM}$ and the existence of a right order on a free monoid satisfying certain corresponding inequalities.

\begin{proposition}\label{p:rovalidmonone}
Let $\eps=(\bigwedge_{i=1}^n t_{i}\leq \bigvee_{j=1}^m s_{j})$ where $s_j,t_i\in{F_{m}(X)}$ \thinspace for $1\le{i}\le{n}$,  $1\le{j}\le{m}$. Then $\vty{DLM}\models\eps$ if and only if there is no right order $\le$ on $\m{F}_m(X)$ satisfying $s_j < t_i$ \thinspace for $1\le{i}\le{n}$, $1\le{j}\le{m}$.
\end{proposition}
\begin{proof}
For the  left-to-right direction, suppose contrapositively that there exists a right order $\le$ on $\m{F}_m(X)$ satisfying $s_j < t_i$ for $1\le{i}\le{n}$,  $1\le{j}\le{m}$. Then $\vty{DLM}\not\models\eps$ by Lemma~\ref{l:notvalidimpliespreorder}(a). For the converse,  suppose contrapositively that $\vty{DLM}\not\models\eps$. By Theorem~\ref{t:eqtheory}, also $\vty{LG}\not\models\eps$ and, rewriting the equation, 
\[
\vty{LG} \not\models \e \le \bigvee \{s_j\iv{t}_i\mid 1\le{i}\le{n},\,1\le{j}\le{m}\}.
\]
By Theorem~\ref{t:freegroupsandrightorders}, there exists a right order $\le$ on $\m{F}_g(X)$ such that $s_j\iv{t}_i<\e$, or equivalently $s_j < t_i$, for  $1\le{i}\le{n}$, $1\le{j}\le{m}$. The restriction of $\le$ to $F_m(X)$ therefore provides the required right order on $\m{F}_m(X)$.
\end{proof}

Proposition~\ref{p:rovalidmonone} relates the validity of an equation in $\vty{DLM}$ to the existence of a right order extending an associated set of inequalities on a free monoid. However, it does not relate the existence of a right order on a free monoid extending a given set of inequalities to the validity of some equation in $\vty{DLM}$. The next result establishes such a relationship via the introduction of finitely many new variables.

\begin{theorem}\label{p:rovalidmon}
Let  $s_1,t_1\dots, s_n,t_n\in{F_{m}(X)}$. The following are equivalent:
\begin{enumerate}
\item[\rm (1)]	There exists a right order $\le$ on $\m{F}_g(X)$ satisfying $s_i < t_i$ \thinspace for $1\le{i}\le{n}$.
\item[\rm (2)]	There exists a right order $\le$ on $\m{F}_m(X)$ satisfying $s_i < t_i$ \thinspace for $1\le{i}\le{n}$.
\item[\rm (3)]	$\vty{DLM}\not\models\bigwedge_{i=1}^n{t_iy_i}\le\bigvee_{i=1}^n s_iy_i$ \thinspace for any distinct $y_1, \dots, y_n\not\in X$.
\end{enumerate}
\end{theorem}
\begin{proof}
(1)\,$\Rightarrow$\,(2). This follows directly from the fact that if $\le$ is a right order on $\m{F}_g(X)$, then the restriction of $\le$ to $F_m(X)$ is a right order on $\m{F}_m(X)$.

(2)\,$\Rightarrow$\,(3).  Let $\le$ be a right order on $\m{F}_m(X)$ satisfying $s_i < t_i$ for $1\le{i}\le{n}$, assuming without loss of generality that $X$ is finite. By Lemma~\ref{l:preorderimpliesnotvalid}(b), there exists a homomorphism $\p\colon\m{T}_\ell(X)\to\m{Aut(\langle\Q,\le\rangle)}$ and $q\in\Q$ such that $(q)\p_{s_i} < (q)\p_{t_i}$ for $1\le i\le n$. So $\m{Aut(\langle\Q,\le\rangle)}\not\models \e\le \bigvee_{i=1}^n s_it^{-1}_i$ and clearly $\vty{LG}\not\models \bigwedge_{i=1}^n{t_it^{-1}_i}\le\bigvee_{i=1}^n s_it^{-1}_i$. But then for any distinct $y_1, \dots, y_n\not\in X$, we have $\vty{LG}\not\models \bigwedge_{i=1}^n{t_iy_i}\le\bigvee_{i=1}^n s_iy_i$ and therefore also $\vty{DLM}\not\models\bigwedge_{i=1}^n{t_iy_i}\le\bigvee_{i=1}^n s_iy_i$.
 
(3)\,$\Rightarrow$\,(1). Suppose that $\vty{DLM}\not\models\bigwedge_{i=1}^n{t_iy_i}\le\bigvee_{i=1}^n s_iy_i$ for some distinct $y_1, \dots, y_n \not\in X$. By Theorem~\ref{t:eqtheory}, also $\vty{LG}\not\models\bigwedge_{i=1}^n{t_iy_i}\le\bigvee_{i=1}^n s_iy_i$ and, by multiplying by the inverse of the left side, $\vty{LG} \not\models\e\le(\bigvee_{i=1}^n s_iy_i)(\bigvee_{i=1}^n \iv{y}_i\iv{t}_i)$. But then, since $\vty{LG}\models \bigvee_{i=1}^n s_i\iv{t}_i\le(\bigvee_{i=1}^n s_iy_i)(\bigvee_{i=1}^n \iv{y}_i\iv{t}_i)$, it follows that $\vty{LG} \not\models \e \le\bigvee_{i=1}^n s_i\iv{t}_i$. Hence, by Theorem~\ref{t:freegroupsandrightorders}, there exists a right order $\le$ on $\m{F}_g(X)$ satisfying $s_i\iv{t}_i<\e$, or equivalently $s_i < t_i$, for $1\le{i}\le{n}$. 
\end{proof}

For any group $\m{G}$ and $N\subseteq G$, there exists a right order $\le$ on $\m{G}$ satisfying $a<\e$ for all $a\in N$ if and only if for every finite subset $N'\subseteq N$, there exists a right order $\le'$ on $\m{G}$ satisfying $a<\e$ for all $a\in N'$ (see, e.g.,~\cite[Chapter~5, Lemma~1]{KM94}). Theorem~\ref{p:rovalidmon} therefore yields the following corollary:

\begin{corollary}\label{c:rofreemonfreegr}
Every right order on the free monoid over a set $X$ extends to a right order on the free group over $X$.
\end{corollary}

\noindent
Note also that by left-right duality, every left order on the free monoid over a set $X$ extends to a left order on the free group over $X$.

We conclude this section with a brief discussion of the relationship between distributive $\ell$\nbd{-}monoids and right-orderable monoids. It was proved in~\cite{Hol65} that a group is right-orderable if and only if it is a subgroup of the group reduct of an $\ell$\nbd{-}group, and claimed in~\cite{AE84} that an analogous theorem holds in the setting of  distributive $\ell$\nbd{-}monoids. Indeed, any monoid $\m{M}$ that admits a right order $\le$ embeds into the monoid reduct of the distributive $\ell$\nbd{-}monoid $\m{End}(\langle M,\le\rangle)$ by mapping each $a\in M$ to the order-endomorphism $x\mapsto xa$. However, contrary to the claim made in~\cite{AE84}, it is not the case that every submonoid of the monoid reduct of a distributive $\ell$\nbd{-}monoid is right-orderable. 

\begin{proposition}\label{p:notorderable}
The monoid reduct of $\mathbf{End}(\langle \Om,\leq \rangle)$ is not right-orderable for any chain $\langle \Om,\leq \rangle$ with $\lvert \Om \rvert \geq 3$.
\end{proposition}
\begin{proof}
We first prove the claim for the distributive $\ell$\nbd{-}monoid $\m{End}(\m{3})$ of order-endomorphisms of the three-element chain $\m{3}=\langle\{0,1,2\},\le\rangle$, using the same notation for endomorphisms as in Example~\ref{ex:endtoaut}. Assume towards a contradiction that $\m{End}(\m{3})$ admits a right order $\le$. Note that for any $a,b,c\in End(\m{3})$, if $ba<ca$, then $b<c$, since otherwise $c\leq b$ would yield $ca \leq ba$. Suppose first that $\langle 0,0,2 \rangle < \langle 0,1,1\rangle$. Then 
\[
\langle 0,0,1 \rangle = \langle 0,0,2\rangle \circ \langle 0,1,1\rangle \leq \langle 0,1,1\rangle \circ\langle 0,1,1\rangle = \langle 0,1,1\rangle
\]
and $\langle 0,0,1\rangle \circ \langle 0,1,1\rangle =  \langle 0,0,1 \rangle < \langle 0,1,1\rangle = \langle 0,1,2\rangle \circ\langle 0,1,1\rangle $. So $\langle 0,0,1\rangle < \langle 0,1,2\rangle$, yielding $\langle 0,0,0 \rangle = \langle 0,0,1\rangle \circ \langle 0,0,1\rangle \leq \langle 0,1,2\rangle \circ\langle 0,0,1\rangle = \langle 0,0,1\rangle$. But $\langle 0,0,2 \rangle < \langle 0,1,1\rangle$ also implies $\langle 0,0,1 \rangle = \langle 0,0,2\rangle \circ \langle 0,0,1\rangle \leq \langle 0,1,1\rangle \circ\langle 0,0,1\rangle = \langle 0,0,0\rangle$. Hence $\langle 0,0,1\rangle = \langle 0,0,0\rangle$, a contradiction. By replacing $<$ with $>$ in the above argument,  $\langle 0,0,2 \rangle > \langle 0,1,1\rangle$ implies $\langle 0,0,1\rangle = \langle 0,0,0\rangle$, also a contradiction. So the monoid reduct of $\m{End}(\m{3})$ is not right-orderable.

Now let $\langle \Om,\leq \rangle$ be any chain with $\lvert \Om \rvert \geq 3$. Without loss of generality we can assume that $\m{3}$ is a subchain of $\Om$. We define a map $\f\colon End(\m{3}) \to End(\langle \Om,\leq \rangle)$ by fixing for each $q\in\Om$,
\[
(q)\f_f := 
\begin{cases}
(\lfloor q \rfloor) f& \text{if } 0\leq q\\
q                    & \text{if } q < 0,
\end{cases}
\]
where $\lfloor q \rfloor := \max \{ k\in \set{0,1,2} \mid k\leq q \}$. Observe that $\lfloor \cdot \rfloor$ is order-preserving, so $\f_f \in End(\langle \Om,\leq \rangle)$ for every $f\in  End(\m{3})$. Also $\f$ is injective, since $\f_f$ restricted to $\m{3}$ is $f$ for each $f\in  End(\m{3})$. Let $f,g \in End(3)$ and $q\in \Om$. If $q<0$, then $(q)\f_{f\circ g} = q = (q)(\f_f\circ \f_g)$. Otherwise $0\leq q$, so $(q)\f_{f\circ g} = ((\lfloor q \rfloor)f)g  =  (\lfloor(\lfloor q \rfloor)f\rfloor)g  =  (q)(\f_f\circ \f_g)$. Hence $\f$ is a semigroup embedding. Finally, since the monoid reduct of $\mathbf{End}(\m{3})$ is not right-orderable, it follows that the monoid reduct of $\mathbf{End}(\langle \Om,\leq \rangle)$ is not right-orderable.
\end{proof}

Note that, although a group is left-orderable if and only if it is right-orderable, this is not the case in general for monoids, even when they are submonoids of groups~\cite{Weh21}. Nevertheless, a very similar argument to the one given in the proof of Proposition~\ref{p:notorderable} shows that also the monoid of endomorphisms of any chain with at least three elements cannot be left-orderable. 


\section{From $\ell$\nbd{-}groups to distributive $\ell$\nbd{-}monoids}\label{s:from_lgroups_to_distributive_lmonoids}

The validity of an equation in the variety of Abelian $\ell$\nbd{-}groups is equivalent to the validity of the inverse-free equation obtained  by multiplying on both sides to remove inverses. Although this method fails for $\vty{LG}$, we show here that inverses can still be effectively eliminated from equations, while preserving validity, via the introduction of new variables. Hence, by Theorem~\ref{t:eqtheory}, the validity of an equation in $\vty{LG}$ is equivalent to the validity of finitely many effectively constructed inverse-free equations in $\vty{DLM}$ (Theorem~\ref{t:inversefree}).

The following lemma shows how to remove one occurrence of an inverse from an equation while preserving validity in $\vty{LG}$.

\begin{lemma}\label{l:density}
Let $r,s,t,u,v\in{T_{\ell}(X)}$ and $y\not\in{X}$.
\begin{enumerate}
\item[\rm (a)] $\vty{LG}\models\e \le v \jn st\iff\vty{LG}\models\e \le v \jn sy \jn \iv{y}t$.
\item[\rm (b)] $\vty{LG} \models u \le v \jn s\iv{r}t \iff \vty{LG} \models ryu \le ryv \jn rysyu \jn t$.
\end{enumerate}
\end{lemma}
\begin{proof}
The left-to-right direction of (a) follows from the validity in $\vty{LG}$ of the quasiequation $\e \le xy \jn z \,\Longrightarrow\,\e \le x \jn y \jn z$ (cf.~\cite[Lemma 3.3]{GM16}). For the converse, suppose  that $\vty{LG}\not\models\e \le v \jn st$. Then $\m{Aut}(\langle\Q,\le\rangle)\not\models\e \le v \jn st$, by Theorem~\ref{t:AutQgeneratesLG}. Hence there exist a homomorphism $\f \colon \m{T}_{\ell}(X) \to \m{Aut}(\langle\Q,\le\rangle)$ and $q \in \Q$ such that $(q)\f_v<q$ and $(q)\f_{st} < q$. Consider $p_1,p_2\in\Q$ with $p_1<q<p_2$. Since $(q)\f_s <(q)\f_{\iv{t}}$ and $p_1< p_2$, there exists a partial order-embedding on $\Q$ mapping $(q)\f_s$ to $p_1$ and $(q)\f_{\iv{t}}$ to $p_2$ that extends to an order-preserving bijection $\widehat{\f}_y \in {\rm Aut}(\langle\Q,\le\rangle)$. Now let also $\widehat{\f}_x:=\f_x$ for each $x\in X$ to obtain a homomorphism $\widehat{\f} \colon \m{T}_{\ell}(X \cup \{y\})\to \m{Aut}(\langle\Q,\le\rangle)$ satisfying $q > (q)\widehat{\f}_v$, $q>(q)\widehat{\f}_{sy}$, and $q>(q)\widehat{\f}_{\iv{y}t}$. Hence $\vty{LG}\not\models\e \le v \jn sy \jn \iv{y}t$ as required.

For (b), we apply (a) to obtain
\begin{align*}
\vty{LG} \models u \le v \jn s\iv{r}t\:
\iff\: & \vty{LG} \models\e \le v\iv{u} \jn s\iv{r}t\iv{u}\\
\iff\: & \vty{LG} \models\e \le v\iv{u} \jn sy \jn \iv{y}\iv{r}t\iv{u}\\
\iff\: & \vty{LG} \models ryu \le ryv \jn rysyu \jn t.\qedhere
\end{align*}
\end{proof}

Eliminating variables as described in the proof of Lemma~\ref{l:density} yields an inverse-free equation that is valid in $\vty{LG}$ if and only if it is valid in $\vty{DLM}$.

\begin{theorem}\label{t:inversefree}
Let $\eps$ be any $\ell$\nbd{-}group equation with variables in a set $X$. A finite set of inverse-free equations $\Si$ with variables in $X\cup Y$ for some finite set $Y$ can be effectively constructed such that $\eps$ is valid in all $\ell$\nbd{-}groups if and only if the equations in $\Si$ are valid in all distributive $\ell$\nbd{-}monoids.
\end{theorem}
\begin{proof}
Let $\eps$ be any equation with variables in a set $X$. Since $\vty{LG}\models s\eq t$ if and only if $\vty{LG}\models \e\le \iv{s}t\mt s\iv{t}$ and every $\ell$\nbd{-}group term is equivalent in $\vty{LG}$ to a meet of joins of group terms, we may assume that $\eps$ has the form $\e\le u_1\mt\cdots\mt u_k$ for some joins of group terms $u_1,\dots,u_k$. Suppose now that for each $i\in\{1,\dots,k\}$, a finite set of inverse-free equations $\Si_i$ with variables in $X\cup Y_i$ for some finite set $Y_i$ can be effectively constructed such that $\e\le u_i$ is valid in all $\ell$\nbd{-}groups if and only if the equations in $\Si_i$ are valid in all distributive $\ell$\nbd{-}monoids. Then $\Si:=\Si_1\cup\cdots\cup\Si_k$ with variables in $X\cup Y$, where $Y:=Y_1\cup\cdots\cup Y_k$ is the finite set of inverse-free equations required by the theorem.

Generalizing slightly for the sake of the proof, it therefore suffices to define an algorithm that given as input any $t_0\in{T}_m(X)$ and $t_1, \dots, t_n\in{T}_g(X)$ constructs $s_0,s_1,\dots,s_m\in{T}_m(X\cup Y)$ for some finite set $Y$ such that
\[
\vty{LG}\models{t_0 \le t_1 \jn \cdots \jn t_n}\iff\vty{DLM}\models{s_0 \le s_1 \jn \cdots \jn s_m}.
\]
If $t_0 \le t_1 \jn \cdots \jn t_n$ is an inverse-free equation, then the algorithm outputs the same equation, which satisfies the equivalence by Theorem~\ref{t:eqtheory}. Otherwise, suppose without loss of generality that $t_1 = u\iv{x}v$. By Lemma~\ref{l:density}, for any $y\not\in X$,
\[
\vty{LG}\models{t_0 \le t_1 \jn \cdots \jn t_n}\iff\vty{LG}\models{xyt_0 \le xyuyt_0 \jn v \jn xyt_2 \jn\cdots \jn xyt_n}.
\]
The equation $xyt_0 \le xyuxt_0 \jn v \jn xyt_2 \jn\cdots \jn xyt_n$ contains fewer inverses than $t_0 \le t_1 \jn \cdots \jn t_n$, so iterating this procedure produces an inverse-free equation after finitely many steps.
\end{proof}

\noindent 
Since the variety $\vty{DLM}$ has the finite model property (Theorem~\ref{t:fmp}), the algorithm given in the proof of Theorem~\ref{t:inversefree} provides an alternative proof of the decidability of the equational theory of $\ell$\nbd{-}groups, first established in~\cite{HM79}.


\section{Totally ordered monoids}\label{s:representable_distributive_lmonoids}

In this section, we turn our attention to totally ordered monoids and groups, that is, distributive $\ell$\nbd{-}monoids and $\ell$\nbd{-}groups where the lattice order is total. We show that the variety generated by the class $\vty{OM}$ of totally ordered monoids can be axiomatized relative to $\vty{DLM}$ by a single equation (Proposition~\ref{p:representabledlm}), and that there exist inverse-free equations that are valid in the class $\vty{OG}$ of totally ordered groups but not in $\vty{OM}$ (Theorem~\ref{t:subrvarrep}). We also prove that there is an inverse-free equation that is valid in all finite totally ordered monoids, but not in the ordered group of the integers (Proposition~\ref{p:fmpfails}), showing that the variety of commutative distributive $\ell$\nbd{-}monoids and the varieties generated by totally ordered monoids and inverse-free reducts of totally ordered groups do not have the finite model property (Corollary~\ref{c:fmpfails}). The proofs of these results build on earlier work on  distributive $\ell$\nbd{-}monoids by Merlier~\cite{Mer71} and Repnitski\u{\i}~\cite{Rep83,Rep84}.

We begin by establishing a subdirect representation theorem for distributive $\ell$\nbd{-}monoids. Note first that since every distributive $\ell$\nbd{-}monoid $\m{M}$ has a distributive lattice reduct, prime ideals of its lattice reduct exist. For a prime (lattice) ideal $I$ of a distributive $\ell$\nbd{-}monoid $\m{M}$ and $a,b\in M$, define
\[
\idq{I}{a}:=\{\langle c,d\rangle \in M \times M\mid cad \in I\}
\quad\text{and}\quad 
a \sim_I b\defiff \idq{I}{a}=\idq{I}{b}.
\]

\begin{proposition}[\cite{Mer71}]\label{prop:kerntwo}
Let $\m{M}$ be a distributive $\ell$\nbd{-}monoid and let $I$ be a prime lattice ideal of $\m M$. Then $\sim_I$ is an $\ell$\nbd{-}monoid congruence and the quotient $\m M/I := \m{M}/{\sim_I}$ is a distributive $\ell$\nbd{-}monoid. Moreover, for any $a,b\in M$, 
\[
[a]_{\sim_I} \leq [b]_{\sim_I} \iff \idq{I}{b} \subseteq \idq{I}{a}, \quad
\idq{I}{a \jn b}=\idq{I}{a} \cap \idq{I}{b}, \quad\text{and}\quad   
\idq{I}{a \mt b}=\idq{I}{a} \cup \idq{I}{b}.
\]
In particular, $\m M/I$ is totally ordered if and only if $\langle\{\idq{I}{a}\mid a\in M\},\subseteq\rangle$ is a chain.
\end{proposition} 

\begin{proposition}\label{p:mainembeddingdlm}
Every distributive $\ell$\nbd{-}monoid $\m M$ is a subdirect product of all the distributive $\ell$\nbd{-}monoids of the form $\m M/I$, where $I$ is a prime ideal of $\m M$. 
\end{proposition}
\begin{proof}
Let $\mathcal{I}$ be the set of all prime lattice ideals of $\m M$. By Proposition~\ref{prop:kerntwo}, there exists a natural surjective homomorphism $\nu_I \colon \m{M} \to \m{M} /I;\: a\mapsto [a]_{\sim_I}$ for each $I \in \mathcal{I}$. Combining these maps, we obtain a homomorphism
\[
\nu\colon \m{M} \to \prod_{I \in \mathcal{I}} \m{M}/I; \quad a \mapsto (\nu_I(a))_{I \in \mathcal{I}}.
\]
It remains to show that $\nu$ is injective. Let $a,b\in M$ with $a\neq b$. By the prime ideal separation theorem for distributive lattices, there exists an $I \in \mathcal{I}$ such that, without loss of generality, $a\in I$ and $b\notin I$, yielding $\langle \e,\e\rangle \in \idq{I}{a}$ and $\langle \e,\e\rangle \notin\idq{I}{b}$. But then $\nu_I(a) \neq \nu_I(b)$ and $\nu(a) \neq \nu(b)$. So $\nu$ is a subdirect embedding.
\end{proof}

The following lemma provides a description of the prime lattice ideals $I$ of a distributive $\ell$\nbd{-}monoid $\m{M}$ such that $\m M/I$ is a totally ordered monoid.

\begin{lemma}\label{l:lin}
Let $\m{M}$ be a distributive $\ell$\nbd{-}monoid and let $I$ be a prime lattice ideal of $\m{M}$. Then $\m M/I$ is totally ordered if and only if for all $b_1,b_2,c_1,c_2,d_1,d_2 \in M$,
\[
c_1b_1c_2 \in I \:\text{ and }\:d_1b_2d_2 \in I \enspace\Longrightarrow\enspace c_1b_2c_2 \in I\:\text{ or }\:d_1b_1d_2 \in I.
\]
\end{lemma}
\begin{proof}
Suppose first that $\m M/I$ is totally ordered and hence, by Proposition~\ref{prop:kerntwo}, that $\idq{I}{b_1} \subseteq \idq{I}{b_2}$ or $\idq{I}{b_2} \subseteq \idq{I}{b_1}$ for all $b_1,b_2 \in M$. Then $c_1b_1c_2 \in I$ (i.e., $\langle c_1,c_2\rangle \in \idq{I}{b_1}$) and $d_1b_2d_2 \in I$ (i.e., $\langle d_1,d_2\rangle \in \idq{I}{b_2}$) must entail $c_1b_2c_2 \in I$ (i.e., $\langle c_1,c_2\rangle \in \idq{I}{b_2}$) or $d_1b_1d_2 \in I$ (i.e., $\langle d_1,d_2\rangle\in \idq{I}{b_1}$) as required. For the converse, suppose that $\m M/I$ is not totally ordered. By Proposition~\ref{prop:kerntwo}, there exist $b_1,b_2\in M$ such that $\idq{I}{b_1} \not \subseteq \idq{I}{b_2}$ and $\idq{I}{b_2} \not \subseteq \idq{I}{b_1}$. That is, there exist $c_1,c_2,d_1,d_2 \in M$ such that $c_1b_1c_2 \in I$ and $d_1b_2d_2 \in I$, but $c_1b_2c_2 \not\in I$ and $d_1b_1d_2 \not\in I$, as required.
\end{proof}

An $\ell$\nbd{-}group or a distributive $\ell$\nbd{-}monoid is called {\em representable} if it is isomorphic to a subdirect product of members of $\vty{OG}$ or $\vty{OM}$, respectively. The following result provides a characterization of representable distributive $\ell$\nbd{-}monoids in terms of their prime lattice ideals, and an equation axiomatizing the variety of these algebras relative to $\vty{DLM}$.

\begin{proposition}\label{p:representabledlm}
The following are equivalent for any distributive $\ell$\nbd{-}monoid $\m{M}$:
\begin{enumerate}
\item [\rm (1)] $\m{M}$ is representable.
\item [\rm (2)] $\m{M}\models (x_1 \le x_2 \jn z_1y_1z_2)\,\&\,(x_1 \le x_2 \jn w_1y_2w_2)\,\Longrightarrow\, x_1 \le x_2 \jn z_1y_2z_2 \jn w_1y_1w_2$.
\item [\rm (3)] $\m{M}\models z_1y_1z_2 \mt w_1y_2w_2 \le z_1y_2z_2 \jn w_1y_1w_2$.
\item [\rm (4)] For any prime lattice ideal $I$ of $\m{M}$, the quotient $\m M/I$ is totally ordered.
\end{enumerate}
\end{proposition}
\begin{proof}
(1)\,$\Rightarrow$\,(2). Since quasiequations are preserved by taking direct products and subalgebras,  it suffices to prove that (2) holds for the case where $\m{M}$ is a totally ordered monoid. Let $a_1,a_2,b_1,b_2,c_1,c_2,d_1,d_2 \in M$ satisfy $a_1 \le a_2 \jn c_1b_1c_2$ and $a_1 \le a_2 \jn d_1b_2d_2$. Since $\m{M}$ is totally ordered, we can assume without loss of generality that $b_1 \le b_2$. It follows that $c_1 b_1c_2 \le c_1b_2c_2$ and therefore $a_1 \le a_2 \jn  c_1b_1c_2  \le  a_2 \jn  c_1b_2c_2 \le   a_2 \jn  c_1b_2c_2  \jn d_1b_1d_2$ as required.

(2)\,$\Rightarrow$\,(3). Let $s_1 := z_1y_1z_2$, $s_2 := w_1y_2w_2$, $t_1 := z_1y_2z_2$, and $t_2 := w_1y_1w_2$, and suppose that $\m{M}\models (x_1\le x_2\jn s_1) \,\&\, (x_1 \le x_2 \jn s_2)\,\Longrightarrow\, x_1 \le x_2 \jn t_1 \jn t_2$. Since $\m{M}\models s_1 \mt s_2 \le t_1 \jn t_2 \jn s_1$ and  $\m{M}\models s_1 \mt s_2 \le t_1 \jn t_2 \jn s_2$, it follows that $\m{M}\models s_1 \mt s_2\le t_1 \jn t_2$ as required.

(3)\,$\Rightarrow$\,(4). Assume (3) and suppose that $c_1b_1c_2 \in I$ and $d_1b_2d_2 \in I$ for some $b_1,b_2,c_1,c_2,d_1,d_1 \in M$. Since $I$ is a lattice ideal, $c_1b_1c_2 \jn d_1b_2d_2 \in I$. By (3) and the downwards closure of $I$, also $c_1b_2c_2 \mt d_1b_1d_2 \in I$. But then, since $I$ is prime, it must be the case that either $c_1b_2c_2 \in I$ or $d_1b_1d_2 \in I$. Hence, by Lemma~\ref{l:lin}, the quotient $\m M/I$ is totally ordered. 

(4)\,$\Rightarrow$\,(1). By (4), $\m M /I$ is totally-ordered when $I$ is a prime ideal of $\m M$, so representability follows by Proposition~\ref{p:mainembeddingdlm}.
\end{proof}

It follows directly from Propositions~\ref{p:mainembeddingdlm} and~\ref{p:representabledlm} that the class of representable distributive $\ell$\nbd{-}monoids is the variety generated by the class $\vty{OM}$ of totally ordered monoids. Similarly, it follows from these results that the class of representable $\ell$\nbd{-}groups is  the variety generated by the class $\vty{OG}$ of totally ordered groups and is axiomatized relative to $\vty{LG}$ by $z_1y_1z_2 \mt w_1y_2w_2 \le z_1y_2z_2 \jn w_1y_1w_2$. (Just observe that if the inverse-free reduct of an $\ell$\nbd{-}group $\m{L}$ is a subdirect product of totally ordered monoids, then each component is a homomorphic image of $\m{L}$ and hence a totally ordered group.) Hence, an equation is valid in these varieties if and only if it is valid in their totally ordered members.

We also obtain the following known fact:

\begin{corollary}[{\cite[Corollary 2]{Mer71}}] 
Commutative distributive $\ell$\nbd{-}monoids are representable.
\end{corollary}
\begin{proof}
By Proposition~\ref{p:representabledlm}, it suffices to note that for any commutative distributive $\ell$\nbd{-}monoid $\m{M}$  and $b_1,b_2,c_1,c_2,d_1,d_2 \in M$, 
\begin{align*}
c_1b_1c_2 \mt d_1b_2d_2 & = c_1c_2b_1 \mt d_1d_2b_2 \\ 
& \le (c_1c_2 \jn d_1d_2)b_1 \mt (c_1c_2 \jn d_1d_2)b_2 \\ 
& = (c_1c_2 \jn d_1d_2)(b_1 \mt b_2) \\ 
& = c_1c_2(b_1 \mt b_2) \jn d_1d_2(b_1 \mt b_2) \\ 
& \le  c_1c_2b_2 \jn d_1d_2b_1 \\ 
& = c_1b_2c_2 \jn d_1b_1d_2. \qedhere
\end{align*}
\end{proof}

It is shown in~\cite{Rep83} that there are inverse-free equations that are valid in all totally ordered Abelian groups, but not in all totally ordered commutative monoids. We make use here of just one of these equations.

\begin{lemma}[{\cite[Lemma~7]{Rep83}}]\label{l:commutativedlm}
The following equation is valid in all totally ordered Abelian groups, but not in all  totally ordered commutative monoids:
\[
x_1x_2x_3 \mt x_4x_5x_6 \mt x_7x_8x_9 \le x_1x_4x_7 \jn x_2x_5x_8 \jn x_3x_6x_9.
\]
\end{lemma}

We use this result to show that the same discrepancy holds when comparing the equational theories of $\vty{OM}$ and $\vty{OG}$.

\begin{theorem}\label{t:subrvarrep}
There is an inverse-free equation that is valid in all totally ordered groups, but not in all totally ordered monoids.
\end{theorem}
\begin{proof}
Consider the inverse-free equation $t_1 \mt t_2 \le s_1 \jn s_2$, where
\begin{align*}
t_1 := x_1x_2x_3 \mt x_5x_4x_6 \mt x_9x_7x_8; \quad & s_1 := x_1x_4x_7 \jn x_5x_2x_8 \jn x_9x_6x_3;\\
t_2 := x_1x_3x_2 \mt x_5x_6x_4 \mt x_9x_8x_7; \quad & s_2 := x_1x_7x_4 \jn x_5x_8x_2 \jn x_9x_3x_6.
\end{align*}
Clearly  $t_1\eq t_2$ and $s_1\eq s_2$ are valid in all totally ordered commutative monoids, so $t_1 \mt t_2 \le s_1 \jn s_2$ fails in some totally ordered monoid by Lemma~\ref{l:commutativedlm}. It remains to show that this equation, or equivalently $\e \le (\iv{t}_1 \jn \iv{t}_2)(s_1 \jn s_2)$, is valid in every totally ordered group. Recall first that (cf.~\cite[Lemma 3.3]{GM16})
\begin{align}\label{eq:split}
\vty{LG}\models\e \le xy \jn z\,\Longrightarrow\,\e \le x \jn y \jn z.
\end{align}
Since $\vty{LG}\models\e \le \e\jn x_8\iv{x}_3\iv{x}_8x_3$, it follows using (\ref{eq:split}) that
\begin{align}\label{eq:first}
\vty{LG}\models \e \le \iv{x}_3x_8x_3\iv{x}_8 \jn x_8\iv{x}_3\iv{x}_8x_3.
\end{align}
An application of (\ref{eq:split}) with (\ref{eq:first}) as premise yields
\begin{align}\label{eq:second}
\vty{LG}\models \e \le \iv{x}_3x_8\iv{x}_6x_7 \jn \iv{x}_7x_6x_3\iv{x}_8 \jn x_8\iv{x}_3\iv{x}_8x_3,
\end{align}
and then another application of (\ref{eq:split}) with (\ref{eq:second}) as premise yields
\begin{align}\label{eq:third}
\vty{LG}\models \e \le \iv{x}_3x_8\iv{x}_6x_7 \jn \iv{x}_7x_6x_3\iv{x}_8 \jn x_8\iv{x}_3x_7\iv{x}_6 \jn x_6\iv{x}_7\iv{x}_8x_3.
\end{align}
For any ordered group $\m{L}$ and $a,b,c\in L$, if $\e\le ab\jn c$, then either $\e\le c$, or $\iv{a}\le b$ and hence $\e\le ba$, so $\e\le ba\jn c$. Hence
\begin{align}\label{eq:cycle}
\vty{OG}\models\e \le xy \jn z\,\Longrightarrow\,  \e \le yx \jn z.
\end{align} 
We apply (\ref{eq:cycle}) four times with (\ref{eq:third}) as the first premise to obtain
\begin{align}\label{eq:fourth}
\vty{OG}\models\e \le x_7\iv{x}_3x_8\iv{x}_6 \jn \iv{x}_8\iv{x}_7x_6x_3 \jn \iv{x}_3x_7\iv{x}_6x_8 \jn \iv{x}_7\iv{x}_8x_3x_6.
\end{align}
For convenience, let
\begin{align*}
u_1 & := \iv{x}_3\iv{x}_2x_4x_7; & u_2 & := \iv{x}_6\iv{x}_4x_2x_8; & u_3 & := \iv{x}_8\iv{x}_7x_6x_3; \\
u_4 & := \iv{x}_2\iv{x}_3x_7x_4; & u_5 & := \iv{x}_4\iv{x}_6x_8x_2; & u_6 & := \iv{x}_7\iv{x}_8x_3x_6.
\end{align*}
An application of (\ref{eq:split}) with (\ref{eq:fourth}) as premise yields
\begin{align}\label{eq:fifth}
\vty{OG}\models\e \le  x_7\iv{x}_3\iv{x}_2x_4 \jn \iv{x}_4x_2x_8 \iv{x}_6 \jn u_3 \jn \iv{x}_3x_7\iv{x}_6x_8 \jn u_6.
\end{align}
Applying (\ref{eq:cycle}) twice with (\ref{eq:fifth}) as the first premise, we obtain
\begin{align}\label{eq:sixth}
\vty{OG}\models\e \le u_1 \jn u_2 \jn u_3 \jn \iv{x}_3x_7\iv{x}_6x_8 \jn u_6.
\end{align}
Another application of (\ref{eq:split}) with (\ref{eq:sixth}) as premise yields
\begin{align}\label{eq:seventh}
\vty{OG}\models\e \le  u_1 \jn u_2 \jn u_3 \jn \iv{x}_3x_7x_4\iv{x}_2 \jn x_2\iv{x}_4\iv{x}_6x_8 \jn u_6.
\end{align}
Applying (\ref{eq:cycle}) twice with  (\ref{eq:seventh}) as the first premise, we obtain
\begin{align}\label{eq:eighth}
\vty{OG}\models\e \le u_1\jn u_2\jn u_3\jn u_4\jn u_5\jn u_6.
\end{align}
Observe now that for some joins of group terms $u',u''$,
\[
\vty{OG}\models\iv{t}_1s_1\eq u_1 \jn u_2 \jn u_3 \jn u'
\quad\text{and}\quad
\vty{OG}\models\iv{t}_2s_2\eq u_4 \jn u_5 \jn u_6 \jn u''.
\]
Hence, since $\vty{OG}\models (\iv{t}_1 \jn \iv{t}_2)(s_1 \jn s_2)\eq \iv{t}_1s_1 \jn  \iv{t}_1s_2 \jn \iv{t}_2s_1 \jn  \iv{t}_2s_2$, by (\ref{eq:eighth}),
\begin{align*}
\vty{OG}\models\e \le (\iv{t}_1 \jn \iv{t}_2)(s_1 \jn s_2). & \qedhere
\end{align*}
\end{proof}

In~\cite{Rep83}, it is proved that the variety generated by the class of inverse-free reducts of Abelian $\ell$\nbd{-}groups is not finitely based and can be axiomatized relative to $\vty{DLM}$ by the set of inverse-free equations  $s_1 \mt \cdots \mt s_n \le t_1 \jn \cdots \jn t_n$ such that $s_1,\dots,s_n,t_1,\dots,t_n\in T_m(X)$ and $s_1\cdots s_n\eq t_1\cdots t_n$ is valid in all commutative monoids. It is not known, however, if the variety generated by the class of inverse-free reducts of totally ordered groups is finitely based. Decidability in each case of the equational theories of commutative distributive $\ell$\nbd{-}monoids, totally ordered monoids, and inverse-free reducts of totally ordered groups is also open. The following result shows, at least, that unlike $\vty{DLM}$, the varieties generated by these classes do not have the finite model property.

\begin{proposition}\label{p:fmpfails}
There is an equation that is valid in every finite totally ordered monoid, but not in $\m{Z}=\langle\Z,\min,\max,+,0\rangle$.
\end{proposition}

\begin{proof}
Consider the equation $xy^2 \leq e\jn x^2y^3$. Note that $\mathbf{Z}\not\models xy^2 \leq e\jn x^2y^3$, since $(-3) + 2 + 2 = 1 > 0 = 0 \jn ((-3) + (-3) + 2 + 2 + 2)$. We show that this equation holds in every finite totally ordered monoid $\mathbf{M}$. Suppose towards a contradiction that $ab^2 > e \jn a^2b^3$ for some $a,b\in M$, i.e., $ab^2>e$ and $ab^2 > a^2b^3$. 

Observe first that, inductively, $ab^2 > a^{2+n}b^{3+n}$ for each $n\in \N$. The base case $n=0$ holds by assumption, and for $n>0$, assuming $ab^2 > a^{2+n-1}b^{3+n-1}$ yields $ab^2>a^2b^3 = a(ab^2)b \geq a( a^{2+n-1}b^{3+n-1})b = a^{2+n}b^{3+n}$. Also, inductively, $a^nb^{2n} \geq ab^2$ for each $n\in\N^{>0}$. The base case $n=1$ is clear, and for $n>1$, assuming $a^{n-1}b^{2n-2} \geq ab^2$ yields (recalling that $ab^2 > e$),
\[
a^nb^{2n} = a^{n-1}(ab^2)b^{2n-2} \geq a^{n-1}eb^{2n-2}= a^{n-1}b^{2n-2} \geq ab^2.
\]
Finally, since $\bf M$ is finite and totally ordered, $a^{n+1} = a^n$ and $b^{n+1} = b^n$ for some $n\in\mathbb{N}$. But then  $ab^2 > a^{2+n}b^{3+n}  =a^nb^n =  a^nb^{2n} \geq ab^2$, a contradiction. 
\end{proof}

\begin{corollary}\label{c:fmpfails}
The variety of commutative distributive $\ell$\nbd{-}monoids and varieties generated by the classes of totally ordered monoids and inverse-free reducts of totally ordered groups do not have the finite model property. 
\end{corollary}


\bibliographystyle{plain}


\end{document}